\documentclass[a4paper, 11pt,draft]{amsart}

\usepackage{amsmath,amsthm,amsfonts,amssymb,amscd}
\input{xy}
\xyoption{all}

\theoremstyle{plain}
\newtheorem{theorem}{Theorem}

\newtheorem{lemma}[theorem]{Lemma}
\newtheorem{corollary}[theorem]{Corollary}

\theoremstyle{definition}
\newtheorem{definition}[theorem]{Definition}

\theoremstyle{remark}
\newtheorem{remark}[theorem]{Remark}

\newtheorem*{notation}{Notations}

\newcommand{\p}{{\mathbb P}}

\begin{document}
\title[On homaloidal polynomials]{On homaloidal polynomial functions
of degree 3 and prehomogeneous vector spaces}

\author{Pierre-Emmanuel Chaput}
\address{D{\'e}partement de Math{\'e}matiques
BP 92208
2, Rue de la Houssini{\`e}re
F-44322 Nantes Cedex 03 - France}
\email{chaput@math.univ-nantes.fr}

\author{Pietro Sabatino}
\address{Via Delle Mimose 9 Int. 14, 00172 - Roma Italy}
\email{pietrsabat@gmail.com}



\begin{abstract}
  In this paper we consider homaloidal polynomial functions $f$ such that their
  multiplicative Legendre transform $f_*$, defined as in \cite[Section
  3.2]{ekp},  
  is again polynomial. Following Dolgachev \cite{dolgachev:pct}, we call such
  polynomials EKP-homaloidal. We prove that every EKP-homaloidal 
  polynomial function of degree three
  is a relative invariant of a symmetric prehomogeneous vector space.
  This provides a complete proof of \cite[Theorem 3.10, p.~39]{ekp}.
  With respect to the original argument of Etingof, Kazhdan and Polischuk our
  argument focuses more on prehomogeneous vector spaces and, in principle, it 
  may
  suggest a way to attack the more general problem raised in \cite[Section
  3.4]{ekp} of classification of EKP-homaloidal
  polynomials of arbitrary degree.
\end{abstract}
\maketitle

Let $V$ be a complex vector space of dimension greater or equal to two
and denote by $V^*$ its dual. 
If $f$ is a function on $V$ and $x \in V$, we denote by $f'(x) \in V^*$ the
differential of $f$ at $x$.
Let $f$ be a homogeneous function
defined on $V$ such that the Hessian of $\ln f$ is generically invertible.
There always exists a homogeneous
function $f_*$ defined on $V^*$ satisfying 
\begin{equation}
  f_*\left(\frac{f'(x)}{f(x)}\right)=\frac{1}{f(x)}
  \label{equation:definition}
\end{equation}
on an open subset of $V$ (\cite[p.~193, Lemma 1]{dolgachev:pct}). Following
\cite{ekp} we call $f_*$ the \emph{multiplicative Legendre transform} of
$f$. Assume further that $f$ is polynomial.
It turns out that $f$ is homaloidal, i.e. its polar map
$x \mapsto f'(x)$ induces a
birational map $\mathbb{P}(V) \dasharrow \mathbb{P}(V^*)$, if and only
if $f_*$ is a rational function, see \cite[p.~37, Proposition 3.6]{ekp} or
\cite[p.~193, Theorem 2]{dolgachev:pct}.
By \cite[Proposition 3.5]{ekp}, we have, under the canonical identification
$V=V^{**}$:
\begin{equation}\label{equation:symmetry}
  f_{**}=f,
\end{equation}
\begin{equation}\label{equation:logdifferential}
  \frac{f_*'}{f_*}\circ \frac{f'}{f}={\rm id}_V.
\end{equation}

\begin{definition} 
  Following \cite{dolgachev:pct}, if $f$ is a 
  homogeneous polynomial function such that
  $\mathrm{det}\ \mathrm{Hess}(\ln f)\ne 0$,
  we call it \emph{EKP-homaloidal} if its
  multiplicative Legendre transform $f_*$ is again polynomial. In this case 
  $f_*$ is a homogeneous polynomial function too and $\deg(f)=\deg(f_*)$. 
\end{definition}

In particular every EKP-homaloidal polynomial
function is homaloidal as it follows from
the above discussion.

\begin{notation}
  Let $f$ by an EKP-homaloidal polynomial function.
  We introduce the following notations
  \begin{equation*}
    \mathbb{P}(V)\supseteq V(f)=X\supseteq \mathrm{Sing}(X)=Z
  \end{equation*}
  and
  \begin{equation*}
    \mathbb{P}(V^*)\supseteq V(f_*)=X_*\supseteq \mathrm{Sing}(X_*)=Z_*
  \end{equation*}
  Where by $V(f)$ we denote the hypersurface defined by $f$ in the projective
  space.
  Moreover in what follows
  we are going to denote by $\phi_f$ and $\phi_{f_*}$ the polar maps associated
  to $f$ and $f_*$ respectively.

  Whenever $W \subset \p(V)$ is a projective variety, we denote by $W^\vee$ its
  dual variety, namely the closure of the set of hyperplanes which are tangent
  at a point in $W$. This is a subvariety of $\p(V^*)$.  
  We denote by $Q_f$ the polarization of $f$, that is the only
  three-linear symmetric form on $V$ such that
  $Q_f(A,A,A)=f(A)$, for every $A\in V$.  Note in particular that
  \begin{equation}
    \label{equation:diffpolarization}
    3Q_f(A,A,-)=f'(A)\in V^*
  \end{equation}
  and moreover 
  \begin{equation}
    \label{equation:seconderivatives}
    \frac{\partial^2}{\partial A \partial B}f(x)=6Q_f(A,B,x)
  \end{equation}
  for every $A,B,x\in V$.

  For every $A\in V$ such that $f(A)\neq 0$, denote by
  \begin{equation*}
    \tau_{f, A}:=d_A\left(\frac{f'}{f}\right):V\to V^*
  \end{equation*}
  the linear map given by the differential at $A$ 
  of the rational map
  $\frac{f'}{f}: V\dashrightarrow V^*$, or in other words the
  second logarithmic
  differential of $f$ computed at $A$.
\end{notation}

We refer to \cite{kimura:ipvs} for the definition
of prehomogeneous vector space
and to \cite{bertram} for the definition of prehomogeneous symmetric space.

\begin{theorem}  \label{theorem:principal}
  Let $f$ be an irreducible EKP-homaloidal polynomial
  function of degree three. Then $f$ is
  the relative polynomial invariant of an
  irreducible prehomogeneous symmetric space.
\end{theorem}

\noindent
Recall that if $W \subset {\mathbb P}(V)$ is a projective variety,
its secant variety is by definition the closure of
the union of all the secant lines in ${\mathbb P}(V)$ passing through
two distinct points in $W$.
Recall also \cite[Chapter IV]{zak} the definition and classification of 
Severi varieties.

\begin{remark}
  The following Corollary is stated in \cite{ekp}, but its proof in the form
  provided there, does not seem to be complete. In particular in the proof of
  Proposition 3.16, last two lines of the proof, the argument based on the rank
  of the hessian matrix alone it is not enough to conclude in general that $Z$ is
  smooth.

  Our proof of Corollary \ref{corollary:principal} follows a completely
  different approach than that in \cite{ekp}. We show how to put 
  on $V$ a structure of
  prehomogeneous vector space in which $f$ is a polynomial invariant, 
  this is the content of our Theorem \ref{theorem:principal}. 
  As a consequence we will get Corollary \ref{corollary:principal}. 
  In our opinion this may be helpful in attacking the more general question 
  raised
  in \cite[Section 3.4]{ekp} of classification of EKP-homaloidal polynomial
  functions.
\end{remark}

\begin{corollary} \cite[Theorem 3.10]{ekp}
  If $f$ satisfies the hypothesis of the above Theorem then the singular
  locus of $V(f)$
  is one of the four classical Severi varieties and
  its secant variety is $V(f)$.
  \label{corollary:principal}
\end{corollary}

Before proving Theorem \ref{theorem:principal} and then Corollary
\ref{corollary:principal} we are going to establish a number of preliminary
lemmas.

\begin{lemma}
  \label{lemma:principal1}
  If $f$ is an irreducible EKP-homaloidal polynomial
  function of degree three then 
  \begin{equation*}
    \phi_f(X\setminus Z)\subseteq Z_*, \quad \phi_{f_*}
    (X_*\setminus Z_*)\subseteq Z
  \end{equation*}
  in particular $Z,Z_*\ne \emptyset$.
\end{lemma}

\begin{proof}
  Since $f_*$ resp. $f_*'$ is homogeneous of degree 3 resp. 2,
  \eqref{equation:logdifferential} amounts to the equality
  $$
  \frac{f_*'(f'(x))}{f_*(f'(x))} = \frac {x}{f(x)}\ .
  $$
  By \eqref{equation:definition}, $f_*(f'(x)) = f(x)^2$. Thus,
  $$
  f_*'(f'(x)) = f(x) \cdot x\ .
  $$
  Our claim follows immediately.
\end{proof}

\begin{remark}
  Under the hypotheses of Theorem \ref{theorem:principal}, $f_*$ is irreducible
  too. Indeed by definition of the Legendre transform,
  we have $f_*(f'(x)) = f(x)^2$
  and in case $f_*$ is a reducible polynomial the preceding equality 
  implies that $f$ is reducible too. 
\end{remark}

\begin{lemma}
  \label{lemma:equations-Z}
  The variety $Z$, as a set, is defined by the equations:
  \begin{equation*}
    z \in Z \Longleftrightarrow \forall A \in V, Q_f(A,z,z)=0
  \end{equation*}
\end{lemma}

\begin{proof}
  In fact, $Z$ is by definition the singular locus of $V(f)$, so that (recall 
  \eqref{equation:diffpolarization}) $z$ is in $Z$ if and only if $f'(z)=0$.
\end{proof}

\begin{lemma}
  \label{lemma:principal2}
  For every $A\in V$ such that $f(A)\ne 0$ the map $\tau_{f, A}$ 
  induces a linear
  isomorphism between $\mathbb{P}(V)$ and $\mathbb{P}(V^*)$ which maps $Z$
  into $X^\vee \subseteq Z_*$. An analogous statement holds true for $f_*$.
\end{lemma}

\begin{proof}
  Let $\mathcal{V}$ be the open subset of $V$ whose elements are vectors
  $A\in V$ such that $f(A)\ne 0$ and $Q(A,A,z)\ne 0$ for at least one $z$
  in every irreducible component of $Z$. 
  We claim that $\mathcal{V}$ is not empty. 
  If it were so, let $Z'$ be an irreducible component of $Z$ such that
  $Q(A,A, z^\prime)=0$ for every $A\in V$ and every $z^\prime$ in $Z'$. 
  As a consequence
  $Q(A,B, z^\prime)=0$ for every $A,B\in V$ and $z^\prime \in Z^\prime$. 
  By 
   \eqref{equation:seconderivatives} this would
  imply that every point of $Z^\prime$ is a singular point of $X$ of order 
  three, and $X$ would be a cone. But a cone is never homaloidal, 
  a contradiction.

  If $A\in \mathcal{V}$ then by \eqref{equation:logdifferential}
  $\tau_{f, A}$ is a (linear)
  isomorphism whose inverse is given by $\tau_{f_*, f^\prime/f(A)}$
  (note that $f_*(f^\prime/f(A))= 1/f(A)\ne 0$ by definition of
  $f_*$). 
  We are going to prove that $\tau_{f, A}(Z)\subseteq X^\vee$ now 
  by giving a geometric interpretation of the map
  $\tau_{f, A}$ on an open subset of $Z$.
  Fix a vector $A\in \mathcal{V}$, and consider the subset of $Z$ given by
  \begin{equation*}
    U_A:=\{z\in Z \: |\: Q_f(A,A,z)\neq 0 \}\subset Z\ .
  \end{equation*}
  Since $A\in \mathcal{V}$, $U_A$ is a dense open subset of $Z$.
  By Lemma \ref{lemma:equations-Z},
  \begin{equation}
    Q_f(z,z, A)=0 \:
    \text{ for every } z \in Z
    \text{ and every } A \in V. 
    \label{equation:zeroterm}
  \end{equation}
  Consider the line $\overline{z,A}\subset \mathbb{P}(V)$ passing through
  $z\in U_A$ and the point in $\mathbb{P}(V)$ corresponding to $A$.  
  Using \eqref{equation:zeroterm}, it is easy to check that
  the intersection of $\overline{z,A}\subset \mathbb{P}(V)$ with $X=V(f)$
  consists of the point $z$ (with multiplicity two) and the simple
  point
  \begin{equation*}
    z':= A-\frac{Q_f(A,A, A)}
    {3Q_f(z,A,A)} z =
    A-\frac{f(A)}{f'(A)(z)} z \in
    X\setminus Z\ .
  \end{equation*}
  The embedded tangent hyperplane
  $T_{z'}(X)\in \mathbb{P}(V^*)$ of $X$ at $z'$ is given by the equation
  $f'(z')=0$. By the definition of
  $z'$, \eqref{equation:diffpolarization}
  and \eqref{equation:zeroterm}, we are able to rewrite the equation of
  $T_{z'}(X)$ as
  \begin{equation} \label{equation:tang}
    Q_f(A,A,-)-\frac{2f(A)}
    {f'(A)(z)}Q_f(z,A,-)=0\ .
  \end{equation}
  On the other hand, by computing explicitly
  $d_A\left(\frac{f'}{f}\right)(z)$ in terms of the
  polarization $Q_f$, we get that
  \begin{equation} \label{equation:explitau}
    \tau_{f, A}(z)=3\frac{2f(A)
    Q_f(z,A, -)
    -f'(A)(z)Q_f(A,A,-)}{f(A)^2}\ .
  \end{equation}
  By comparing equations (\ref{equation:tang}) and (\ref{equation:explitau}),
  we get that
  $\tau_{f, A}(z)=T_{z'}(X)$, which gives the required
  geometric interpretation of $\tau_{A, f}$ on the open subset
  $U_A\subset Z$. It follows that 
  $\tau_{f, A}(U_A)\subset X^\vee$ and hence $\tau_{f, A}(Z)\subset X^\vee$. 
  Since this last condition is a closed one it holds true for every
  $A\in V$ such that $f(A)\ne 0$.
\end{proof}

\begin{definition}
  Let $G \subset GL(V)$ be the subgroup preserving $X$:
  \begin{equation*}
    G:=\{ g\in GL(V)\: :\: g(X)=X \}\ .
  \end{equation*}
\end{definition}

\begin{corollary}
  \label{corollary:action}
  \begin{itemize}
    \item[(i)] For each $A\in V \setminus \{ f=0 \}$ 
      we have that $\tau_{A,f}$ is a linear
      isomorphism that sends $Z$ onto $Z_*$ and $X$ onto $X_*$.
    \item[(ii)] $\phi_f(X\setminus Z)=Z_*=X^\vee$ and 
      $\phi_{f_*}(X_*\setminus Z_*)=Z=(X_*)^\vee$.
    \item[(iii)] For every $z_1,z_2 \in Z $ there exist
      $A_1,A_2\in V\setminus \{f=0 \}$ 
      such that 
      \begin{equation*}
	\tau_{f,A_2}^{-1} \circ \tau_{f,A_1}(z_1)=z_2, 
      \end{equation*}
      hence $G$ acts transitively on $Z$, and $Z$, $Z_*$ are smooth. 
    \item[(iv)] The fibers of the Gauss
      map $\phi_f:X \rightarrow Z_*$ are all linear spaces of dimension
      $\dim X - \dim Z_*$ and $G$ acts transitively on them i.e. for every
      couple of distinct fibers of $\phi_f$, $F_1$ and $F_2$, there exists an
      element
      $g\in G$ such that $g(F_1)=F_2$. 
  \end{itemize}
\end{corollary}

\begin{proof}
  We start by proving (i). By Lemma \ref{lemma:principal2} if $A\in V$ and
  $A_*\in V^*$ are such that $f(A)\ne 0$ and $f_*(A_*)\ne 0$ then 
  $\tau_{f,A}(Z)\subseteq X^\vee \subseteq Z_*$ and
  $\tau_{f_*,A^*}(Z_*)\subseteq (X_*)^\vee \subseteq Z$.
  If we
  choose $A\in V \setminus \{f=0 \}$ and $A_* = f'(A)/f(A)$,
  then by definition of $f_*$ 
  we have $A_* \in V^*\setminus \{f_*=0 \}$ and 
  by \eqref{equation:logdifferential}
  \begin{equation*}
    \tau_{f,A} \circ \tau_{f^*,A^*}= \mathrm{Id}_V
  \end{equation*}
  and similarly $\tau_{A_*,f_*} \circ \tau_{f,A} = \mathrm{Id}_{V^*}$.
  This implies $\tau_{f,A}(Z)=Z_*$ and $X^\vee = Z_*$. 
  Moreover $(X_*)^\vee=Z$ in view of the symmetry \eqref{equation:symmetry}. 
  Recalling the geometric interpretation of the maps $\tau_{f,A}$ 
  given in the proof of Lemma \ref{lemma:principal2} this 
  completes the proof of item (ii) too.  

  We are going to prove (iii) now. Let $z_1,z_2\in Z$; by item (ii)  
  there exist points $p_1,p_2\in X_*\setminus Z_*$ such that
  $z_i=T_{p_i}(X_*)$.
  Given $z \in X_*$ we denote by $K_zX_* \subset \mathbb{P} (V^*)$ the tangent
  cone to $X_*$ at $z$, defined by the equivalence
  $u \in K_zX_* \Leftrightarrow Q_{f_*}(z,u,u) = 0$.
  Consider the following closed subsets of $Z_*$, 
  \begin{equation*}
    K_i=\{z \in Z_*\  |\ \overline{z,p_i} \subset K_z X_* \}
  \end{equation*}
  for $i=1,2$. Observe that $K_i\ne Z_*$. Indeed if it were so, the cone
  with vertex $p_i$ and base $Z_*$
  would be contained in $T_{p_i}(X_*)$ and $Z_*$ would be
  degenerate. Then $X$ would be a cone, but a cone is never homaloidal. Put
  $U_i=Z_*\setminus K_i$, it follows that there are points $\xi \in U_1\cap
  U_2\ne \emptyset$ and $B_i\in \mathbb{P}(V^*)\setminus X_*$ such that by the
  geometric interpretation of the maps $\tau_{f_*,-}$ given in the proof 
  of Lemma \ref{lemma:principal2} we have 
  \begin{equation*}
    \tau_{f_*,B_i}(\xi)=z_i
  \end{equation*}
  and hence
  \begin{equation*}
    \tau_{f_*,B_2}\circ \tau_{f_*,B_1}^{-1}(z_1)=z_2
  \end{equation*}
  or in other words 
  \begin{equation*}
    \tau_{f,A_2}^{-1} \circ \tau_{f,A_1}(z_1)=z_2 
  \end{equation*}
  where $A_1=(f_*^\prime/f_*)(B_1)$ and $A_2=(f_*^\prime/f_*)(B_2)$. It follows
  that $G$ acts transitively on $Z$ and $Z$ is smooth. By (iii) this 
  implies that $Z_*$ is smooth too. 

  Concerning item (iv), by the duality theorem, a general fiber of $\phi_f$
  is a linear space of dimension $\dim X - \dim X^\vee$. Since $X^\vee=Z_*$
  is homogeneous under the action of $G$, this is true for any fiber.
\end{proof} 

We are now ready to prove Theorem \ref{theorem:principal}.

\begin{proof}[Proof of Theorem \ref{theorem:principal}]
  We may now argue as in \cite{chaput:scorza} to prove that
  $G$ acts transitively on $V \setminus \{ f=0 \}$
  (see \cite[Lemma 3.1]{chaput:scorza}). 
  Fix an element $I\in V$ such that $f(I)\neq 0$. For any $A\in V$ such
  that $f(A)\neq 0$, consider the map $(\tau_{f, I})^{-1}\circ \tau_{f,
  A}\in GL(V)$. By Lemma \ref{lemma:principal2}, we get that
  $(\tau_{f, I})^{-1}\circ \tau_{f, A}\in G$. Therefore the $G$-orbit of
  $I$ contains the image of the map
  $$\begin{aligned}
    \varphi_I:& V \setminus \{f=0\} \longrightarrow V \\
    A & \mapsto (\tau_{f, I})^{-1}(\tau_{f, A}(I)).
  \end{aligned}$$
  We are going now to compute the differential of $\varphi_I$. First of
  all, observe that the map $(\tau_{f, I})^{-1}(-)$ is linear, hence we
  only need to compute $d_I\big(A\mapsto \tau_{f,A}(I)\big)$.  By
  \eqref{equation:explitau} and \eqref{equation:logdifferential}, a
  straightforward
  computation shows that
  \begin{equation}
    \tau_{f,I+tB}(I)-\tau_{f,I}(I)=-2\tau_{f,I}(B)t+o(t).
    \label{equation:}
  \end{equation}
  Moreover observe that $d_I\big(A\mapsto \tau_{f,A}(I)\big)$ is equal
  to the directional derivative of the map $A\mapsto \tau_{f,A}(I)$ with
  respect to $B$ computed in $I$. It follows that
  \begin{displaymath}
    d_I\big(A\mapsto \tau_{f,A}(I)\big)(B)=-2\tau_{f,I}(B)
  \end{displaymath}
  and then
  \begin{displaymath}
    d_I(\varphi_I)=(\tau_{f,I})^{-1}\circ d_I(A\mapsto \tau_{f,A}(I))=
    -2 \cdot {\rm Id}.
  \end{displaymath}
  Therefore the image of the map $\varphi_I$, and hence the $G$-orbit of
  $I$, is a Zariski dense open subset of $V$ (see \cite{kimura:ipvs} pg.~23
  Lemma 2.1).  Since we can repeat the same argument for all the
  elements $I\in V$ such that $f(I)\neq 0$, we conclude that $G$ acts
  transitively on $V \setminus \{f=0 \}$.

  Moreover the argument in \cite[Proposition 3.1]{chaput:scorza}
  and computations in \cite[Proposition 3.2]{chaput:scorza} show
  that $V$ is indeed a symmetric
  prehomogeneous space i.e. there exists an involution of the group $G$ such
  that the stabilizer of a point in the dense orbit is contained in the set of
  fixed 
  points for the involution and moreover contains the connected component of the
  identity of this set.

  We are going to show that the action of $G$ on $V$ is irreducible now.
  Let $W\subset V$ a
  proper subspace of $V$ invariant for the action of $G$. Since $G$ acts
  transitively on $V\setminus \{f=0 \}$ and $W$ is a proper subspace we must 
  have $W\subset \{f=0 \}$. 
  If $W\cap Z\ne \emptyset$, since $G$ acts transitively on $Z$ too (see
  Corollary \ref{corollary:action} (ii)), then $Z \subset \mathbb {P}(W)$.
  But then its
  dual variety is a cone, so that $V(f_*)$ is a cone, and this contradicts
  the fact that $f_*$ is homaloidal.
  It follows that $W \subset \mathrm{Sm}(X)$, and by Corollary
  \ref{corollary:action} (iv), it intersects every fiber of the Gauss map. It
  follows that every tangent hyperplane to $X$ contains the subspace $W$, since
  every tangent hyperplane is tangent to $X$ in a point of $W$. 
  But then $X$ has degenerate dual, a contradiction.

  The proof of the theorem is now complete.
\end{proof}

Before proceeding to the proof of Corollary \ref{corollary:principal} 
we are going to recall a
result of McCrimmon \cite{mccrimmon}. Let $F$ be an homogeneous polynomial of
degree $d$ on the vector space $V$. Fix a point $I\in V$ such that,
$F(I)\ne 0$ and $\tau_{F,I}$ is a non-degenerate bilinear form. There exists
then a rational mapping
\begin{gather*}
  H: V\rightarrow \mathrm{Hom}(V,V) \\
  A\mapsto H_A
\end{gather*}
defined for all $A\in V$ with $F(A)\ne 0$, such that
\begin{equation*}
  \tau_{F,A}(B,C)= \tau_{F,I}(H_A(B),C)
\end{equation*}
for every $A,B,C\in V$.
\begin{theorem} \cite[Theorem 1.2]{mccrimmon}
  If $F$ satisfies the relation $F(H_A(B))=h(A)Q(B)$ whenever both sides are
  defined, $h$ a suitable rational function, then the product
  \begin{equation}
    A*B=- \frac{1}{2}d_I \big( M\mapsto H_M(B) \big) (A)
    \label{equation:product}
  \end{equation}
  defines on the vector space $V$ a Jordan algebra structure. If moreover we
  choose $I$ such that $F(I)=1$ then 
  \begin{equation}
    F(R_A(B))=F(A)^2F(B)
    \label{equation:composition}
  \end{equation}
  where $R_A$ denotes the quadratic representation of the Jordan algebra.
  \label{theorem:mccrimmon}
\end{theorem}
Recall that the quadratic representation of a Jordan algebra is defined in the
following way. Denote by $L_A$ the linear map induced on $V$ by left
multiplication by an element $A$ of the Jordan algebra. Then 
\begin{equation*}
  R_A:= L_A^2-L_{A^2}\ .
\end{equation*}
\begin{remark}
  An element $A$ in a Jordan algebra is invertible if and only if $R_A$ is an
  invertible linear map, see \cite[Proposition II.3.1, p.~33]{asc}. It follows
  that in the Jordan algebra defined by \eqref{equation:product} if
  \eqref{equation:composition} holds and $\{ F=0  \}$ is not contained in an
  hyperplane 
  then the set of invertible elements 
  coincide with $V\setminus \{ F=0\}$.
  \label{remark:invertibles}
\end{remark}

Finally, we deduce the announced Corollary \ref{corollary:principal} from
Theorem \ref{theorem:principal}.

\begin{proof}[Proof of Corollary \ref{corollary:principal}]
  By Theorem \ref{theorem:principal} above, $V$ is an irreducible
  prehomogeneous symmetric space. 
  Fix an element $I\in V$ such that $f(I)=1$. Observe that the map $H_A$ in
  Theorem \ref{theorem:principal} equals $\tau_{f,I}^{-1}\circ \tau_{f,A}$ and
  then $H_A\in G$. It follows that there exists a rational function $h$ 
  such that $f(H_A(B))=h(A)f(B)$ and then \eqref{equation:product} defines on
  $V$ a structure of Jordan algebra in such a way that
  \eqref{equation:composition} is satisfied. 
  A direct computation, see \cite[p.~179]{chaput:scorza}, 
  shows that the above product
  coincides with the one defined in \cite[p.~43]{bertram}. 
  By \cite[Theorem V.4.6]{bertram} $V$ is a simple Jordan algebra and,
  recall Remark \ref{remark:invertibles}, its
  determinant coincides with a scalar multiple of $f$. 
  By the classification of simple Jordan algebras
  \cite[pp.~204--205, Corollary 2]{jacobson}, the singular locus of $V(f)$ 
  is a Severi variety (see \cite[Chapter IV, Theorem 4.8]{zak}).  
\end{proof}

\bibliography{mybiblio}{}

\begin{thebibliography}{EKP02}

\bibitem[Ber00]{bertram}
Wolfgang Bertram.
\newblock {\em The geometry of Jordan and Lie structures}, volume 1754 of {\em
  Lecture Notes in Mathematics}.
\newblock Springer-Verlag, Berlin, 2000.

\bibitem[Cha03]{chaput:scorza}
Pierre-Emmanuel Chaput.
\newblock Scorza varieties and {J}ordan algebras.
\newblock {\em Indag. Math. (N.S.)}, 14(2):169--182, 2003.

\bibitem[Dol00]{dolgachev:pct}
Igor~V. Dolgachev.
\newblock Polar {C}remona transformations.
\newblock {\em Michigan Math. J.}, 48:191--202, 2000.
\newblock Dedicated to William Fulton on the occasion of his 60th birthday.

\bibitem[EKP02]{ekp}
Pavel Etingof, David Kazhdan, and Alexander Polishchuk.
\newblock When is the {F}ourier transform of an elementary function elementary?
\newblock {\em Selecta Math. (N.S.)}, 8(1):27--66, 2002.

\bibitem[FK94]{asc}
Jacques Faraut and Adam Kor{\'a}nyi.
\newblock {\em Analysis on symmetric cones}.
\newblock Oxford Mathematical Monographs. The Clarendon Press Oxford University
  Press, New York, 1994.
\newblock Oxford Science Publications.

\bibitem[Jac68]{jacobson}
Nathan Jacobson.
\newblock {\em Structure and representations of Jordan algebras}, volume~39 of
  {\em American Mathematical Society Colloquium Publications}.
\newblock American Mathematical Society, Providence, R.I., 1968.

\bibitem[Kim03]{kimura:ipvs}
Tatsuo Kimura.
\newblock {\em Introduction to prehomogeneous vector spaces}, volume 215 of
  {\em Translations of Mathematical Monographs}.
\newblock American Mathematical Society, Providence, RI, 2003.
\newblock Translated from the 1998 Japanese original by Makoto Nagura and
  Tsuyoshi Niitani and revised by the author.

\bibitem[McC65]{mccrimmon}
K.~McCrimmon.
\newblock Norms and noncommutative {J}ordan algebras.
\newblock {\em Pacific J. Math.}, 15:925--956, 1965.

\bibitem[Zak93]{zak}
Fyodor Zak.
\newblock {\em Tangents and secants of algebraic varieties}, volume 127 of {\em
  Translations of Mathematical Monographs}.
\newblock American Mathematical Society, Providence, RI, 1993.
\newblock Translated from the Russian manuscript by the author.

\end{thebibliography}
\bibliographystyle{alpha}
\end{document}